\documentclass[12pt]{article}
\usepackage{graphicx} 
\usepackage{epstopdf}
\usepackage{subfigure}
\usepackage{color}

\usepackage[english]{babel}
\usepackage{amsmath,amsthm}
\usepackage{amssymb}
\usepackage{amsfonts}
\usepackage{booktabs}

\usepackage{graphicx}
\usepackage{subfigure}
\usepackage{tabularx}
\usepackage{bbm}
\usepackage{hyperref}
\hypersetup{
	colorlinks=true,
	linkcolor=blue,
	citecolor=green,
	urlcolor=red,
	breaklinks=true
}

\newtheorem{thm}{Theorem}[section]
\newtheorem{defi}{Definition}[section]

\newtheorem{lem}[thm]{Lemma}
\newtheorem{coro}[thm]{Corollary}

\theoremstyle{plain}

\title{The Alon-Tarsi Number of Cartesian product and Corona product of Hypercube Graph and Special Graphs}

\author { Zhiguo Li{$^*$}, Yujia Gai, Zeling Shao\\
{\small School of Science, Hebei University of Technology, Tianjin 300401, China}
\date{}
}

\begin{document}
\baselineskip 0.65cm

\maketitle

\begin{abstract}
The \emph{Alon-Tarsi number} of a graph $G$ is the smallest $k$ so that there exists an orientation $D$ of $G$ with max outdegree $k-1$ satisfying the number of even Eulerian subgraphs different from the number of odd Eulerian subgraphs. In this paper, the Alon-Tarsi number of the $n$-cube is obtained according to its special properties, we obtain the Alon-Tarsi number of Cartesian product of some special bipartite graphs, and get the Alon-Tarsi number of Corona product of graphs. As corollaries, we get the Alon-Tarsi number of Cartesian product and Corona product of hypercube graph and special graphs.


\noindent\textbf{2000 MR Subject Classification.} 05C15
\end{abstract}

\section{Introduction}

We only consider simple and finite graphs in this paper. The $chromatic$ $number$ of a graph $G$, denoted by $\chi (G)$, is the least positive integer $k$ such that $G$ has a proper vertex coloring using $k$ colors.  List coloring is a well-known variation on vertex coloring. For list coloring, a $k$-$list$ $assignment$ of a graph $G$ is a mapping $L$ which assigns to each vertex $v$ of $G$ a set $L(v)$ of $k$ permissible colors. An $L$-$coloring$ of $G$ is a coloring $f$ of $G$ such that $f(v)\in L(v)$ for each verex $v$. We say $G$ is $L$-$colorable$ if there exists a proper $L$-coloring of $G$. A graph $G$ is $k$-$choosable$ if $G$ is $L$-colorable for every $k$-list assignment $L$. The \emph{choice number} of a graph $G$ is the least positive integer $k$ such that $G$ is $k$-choosable, denoted by $ch(G)$.

A subdigraph $H$ of a directed graph $D$ is called Eulerian if $V(H)=V(D)$ and the indegree $d^{-}_{H}(v)$  of every vertex $v$ in $H$ is equal to its outdegree $d^{+}_{H}(v)$. We do not assume that $H$ is connected.
$H$ is even if the number of arcs of $H$ is even, otherwise, it is odd. Let $\mathcal{E}_{e}(D)$ and $\mathcal{E}_{o}(D)$  denote the
 families of even and odd Eulerian subgraphs of $D$, respectively. Let {\rm diff}$(D)=|\mathcal{E}_{e}(D)|-|\mathcal{E}_{o}(D)|$. We say that $D$ is $Alon$-$Tarsi$ if {\rm diff}$(D)\neq0$. If an $orientation$ $D$ of $G$ yields an $Alon$-$Tarsi$ $digraph$, then we say $D$ is an $Alon$-$Tarsi$ $orientation$ $($or an $AT$-$orientation$, for short$)$ of $G$.

The \emph{Alon-Tarsi number} of $G$ ($AT(G)$, for short) is the smallest $k$ such that there exists an orientation $D$ of $G$ with max outdegree $k-1$ satisfying the number of Eulerian subgraphs with even arcs different from the number of Eulerian subgraphs with odd arcs. It was proposed by Alon and Tarsi$^{[1]}$, subsequently, they used algebraic methods to prove that $\chi(G)\leq ch(G)\leq AT(G)$. The graph $G$ is called $chromatic$-$choosable$ if $\chi(G)= ch(G)$. The graph $G$ is called $chromatic$-$AT~choosable$ if $\chi(G)= AT(G)$.

The graph operation, especially graph product, plays significant role not only in pure and applied mathematics, but also in computer science. For graph product, we are familiar with Cartesian product and Corona product. 

The $Cartesian~product$ of graphs $G$ and $H$, denoted by $G\square H$, is the graph with vertex set $V(G)\times V(H)$  and edges created such that $(u,v)$ is adjacent to $(u^{\prime},v^{\prime})$ if and only if either $u=u^{\prime}$ and $vv^{\prime}\in E(H)$ or $v=v^{\prime}$ and $uu^{\prime}\in E(G)$.
Note that $G\square H$ contains $|V(G)|$ copies of $H$ and $|V(H)|$ copies of $G$.  It is also easy to show that $\chi(G\square H)=\max\{\chi(G),\chi(H)\}$.$^{[2]}$ There are some results concerning the Alon-Tarsi number of Cartesian product of graphs. 

Kaul and Mudrock proved that the Cartesian product of cycle and path has $AT(C_{k}\square P_{n})=3$ in $[2]$ ($k,n$ are integer). Suppose that $G$ is a complete graph or an odd cycle with $|V(G)|\geq3$ and suppose $H$ is a graph on at least two vertices that contains a Hamilton path, $w_1,w_2,\cdots,w_m$, such that $w_i$ has at most $k$ neighbors among $w_1,\cdots,w_{i-1}$. Then, $AT(G\square H)\leq\Delta(G)+k$$^{[2]}$. For any graph $G$, all vertices of maximum degree in which may be covered by some vertex-disjoint cycles, Petrov and Gordeev$^{[10]}$ obtained that $AT(G\square C_{2k})\leq\Delta(G\square C_{2k})-1$. Li, Shao, Petrov and Gordeev in $[3]$ proved that the Cartesian product of a cycle and a cycle has
\[ AT(C_{m}\square C_n)=\begin{cases}
		4 ,  & n~and~m~are~both~odd~numbers;\\
		3,  &  otherwise.
	\end{cases} \]

In this paper, we get the exact value of the Alon-Tarsi number of $n$-cube $Q_n$, and the Alon-Tarsi number of Cartesian product and Corona product between $Q_n$ and special graphs. The main results are the following theorems:\\
{\bf Theorem 1.}
For any $n,m\in \mathbb{N}$, let $T_m$ be a tree with $m$ vertices ($m\geq2$). Then
 \[ AT(Q_{n}\square T_m)=\begin{cases}
		\lceil\frac{n}{2}\rceil+1 ,  & n~is~odd~and~m=2;\\
		\lceil\frac{n}{2}\rceil+2,  &  otherwise.
	\end{cases} \]
{\bf Theorem 2.}
For any $n,m\in \mathbb{N}$, let $G_2$ be a simple graph with $m$ vertices ($m\geq2$). If $AT(G_2)=2$, then
\[ AT(Q_{n}\circ G_2)=\begin{cases}
	3 ,  & n\leq2;\\
	\lceil\frac{n}{2}\rceil+1,  &  n\textgreater2.
\end{cases} \]

 \section{Preliminaries}

\begin{defi}$^{[4]}$
The $n$-cube graph can be defined by Cartesian product denoted by $Q_n$, which is recursive definition as follows: $Q_1=K_2$, $Q_n=K_2\square Q_{n-1}$.
\end{defi}

Let $n\geq1$ be an integer, $B=\{0,1\}$ and $\mathcal{B}_n=\{b_1b_2\cdots b_n|b_i\in B,i\in [n]\},$ where $[n]=\{1,\cdots,n\}$. An element $\beta$ of $\mathcal{B}_n$ is called a binary string of length $n$. The weight of $\beta$ is defined as $w(\beta)=\sum_{i=1}^nb_i$, in other words, $w(\beta)$ is the number of 1s in string $\beta$$^{[5]}$. The two vertices of $Q_1$ are labeled 0 and 1, the second formula in the Definition 2.1 shows that $Q_n$ can be obtained by connecting the corresponding vertices of two copies of $Q_{n-1}$, therefore $Q_2$ consists of two copies of $Q_1$, the vertices of the first copy can be labeled 00 and 10, while the vertices of the second copy are labeled 01 and 11. More generally, the vertex set of $Q_n$ is composed of binary strings of length $n$. The $n$-cube $Q_n$ is the graph defined on the vertex set $\mathcal{B}_n$, and two vertices $b_1b_2\cdots b_n$ and $b_1^{'}b_2^{'}\cdots b_n^{'}$ being adjacent if $b_i\neq b_i^{'}$ holds for exactly one $i\in [n]$. (See Figure 1 for the $n$-cube graph, $n=1,2,3$)

\begin{figure}[htbp]
\centering
\includegraphics[height=6cm, width=0.5\textwidth]{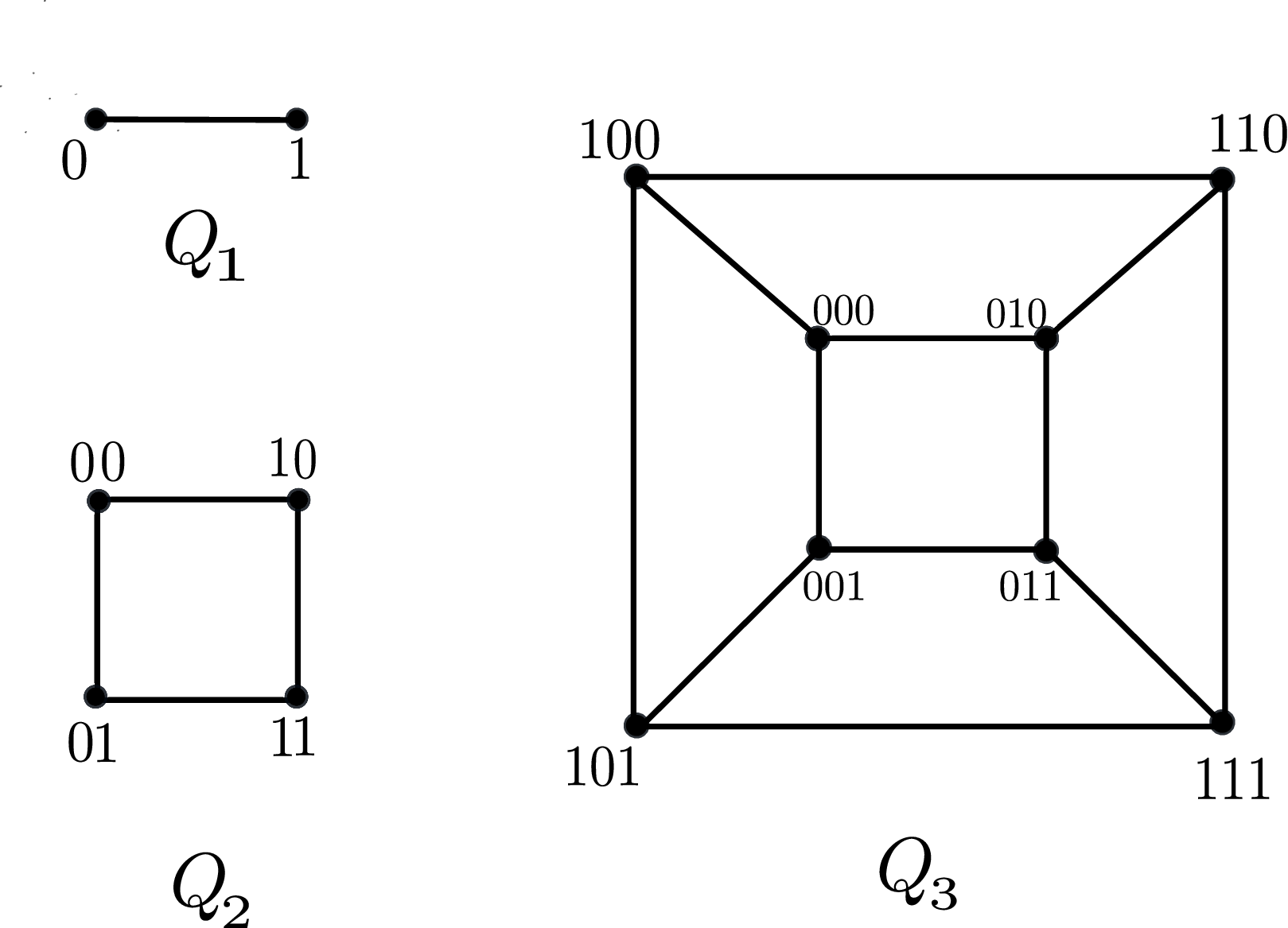}
\centerline{Fig.1 The $n$-cube graph for $n=1,2,3$. }
\end{figure}

It is easy to see $|V(Q_n)|=|\mathcal{B}_n|=2^n$, and each vertex connects $n$ edges, hence the $n$-cube graph is an $n$-regular graph. By Handshake Lemma,
$\sum_{v\in Q_n}d(v)=n2^n=2|E(Q_n)|$, so $|E(Q_n)|=n2^{n-1}.$

\begin{defi}$^{[6]}$
Let $G$ and $H$ be simple graphs. The corona $G\circ H$ of $G$ and $H$ is constructed as follows: Choose a labeling of the vertices of $G$ with labels $1,\cdots,m$. Take one copy of $G$ and $m$ disjoint copies of $H$, labeled $H_1,\cdots,H_m$, and connect each vertex of $H_i$ to vertex $i$ of $G$.
\end{defi}

\begin{lem}$^{[7]}$
If $G$ is a bipartite graph, then every orientation $D$ of $G$ is an $AT$-orientation. And
 $$AT(G)=\max\limits_{\text{subgraph}~H~\text{of}~ G}\lceil\frac{|E(H)|}{|V(H)|}\rceil+1.$$
\end{lem}

For a graph $G$, the \emph{maximum average degree} mad($G$) is defined as the maximum average degree over all subgraphs of $G$, and can be computed by$^{[8]}$
$$\text{mad}(G)=\max\limits_{\text{subgraph}~H~\text{of}~G}\frac{\sum_{v\in V(H)}deg_H(v)}{|V(H)|}=\max\limits_{\text{subgraph}~H~\text{of}~ G}\frac{2|E(H)|}{|V(H)|}.$$

\begin{lem}$^{[9]}$
Assume that $D$ is a digraph and $V(D)=X_1\dot\cup X_2$. For $i=1,2,$ let $D_i=D[X_i]$ be the subdigraph of $D$ induced by $X_i$. If all the arcs between $X_1$ and $X_2$ are from $X_1$ to $X_2$, then $D$ is Alon-Tarsi if and only if $D_1,D_2$ are both Alon-Tarsi.
\end{lem}

\section{The proof of the Main Results}
Firstly, we prove Theorem 1. It will be completed by the following lemmas.
\begin{lem}
If $G$ is a $d$-regular bipartite graph, then $AT(G)=\lceil\frac{d}{2}\rceil+1.$
\end{lem}
\begin{proof}
Since $G$ is a bipartite graph, by Lemma 2.1,
$$AT(G)=\max\limits_{\text{subgraph}~H~\text{of}~ G}\lceil\frac{|E(H)|}{|V(H)|}\rceil+1.$$
According to $G$ is also a $d$-regular graph, and $G$ has maximum average degree $d$, so
$$\max\limits_{\text{subgraph}~H~\text{of}~ G}\frac{|E(H)|}{|V(H)|}=\frac{d}{2}.$$ Then $AT(G)=\lceil\frac{d}{2}\rceil+1.$
\end{proof}

\begin{lem}
For an n-cube graph $Q_n$, $AT(Q_n)=\lceil\frac{n}{2}\rceil+1.$
\end{lem}
\begin{proof}
Note that $Q_n$ is bipartite with partition $(X,Y)$, since two vertices are adjacent if and only if the two vertices differ in one and only one coordinate. Assume $u,v\in V(Q_n)$, if $u$ and $v$ have more than two different coordinates, then both $u$ and $v$ belong to $X$ or $Y$, and if $u$ and $v$ differ in only one coordinate, then $u$ and $v$ belong to different divisions. Since $Q_n$ is an $n$-regular graph, by Lemma 3.1,
$AT(Q_n)=\lceil\frac{n}{2}\rceil+1.$  
\end{proof}
\noindent\textbf{Remark.}
The $n$-cube graph $Q_n$ is not $chromatic$-$AT~choosable$. 

\begin{lem}
Let $G_1$ and $G_2$ be simple bipartite graphs, and $G=G_1\square G_2$ is the Cartesian product of $G_1$ and $G_2$. For the graphs $G_1$ and $G_2$, satisfy $AT(G_1)=\lceil\frac{|E(G_1)|}{|V(G_1)|}\rceil+1,AT(G_2)=\lceil\frac{|E(G_2)|}{|V(G_2)|}\rceil+1$, and $\lceil\frac{|E(G_1)|}{|V(G_1)|}+\frac{|E(G_2)|}{|V(G_2)|}\rceil=\lceil\frac{|E(G_1)|}{|V(G_1)|}\rceil+\lceil\frac{|E(G_2)|}{|V(G_2)|}\rceil,$
then $$AT(G)=AT(G_1)+AT(G_2)-1.$$
\end{lem}
\begin{proof}
 Assume that $G_1$ has partition $(X_1,Y_1)$, and $G_2$ has partition $(X_2,Y_2)$, let $X_2=\{v_1,\cdots,v_m\}$, $Y_2=\{v_{m+1},\cdots,v_{m+n}\}$. Let $G_1$ and $G_2$ be simple bipartite graphs with $|V(G_1)|=p,|E(G_1)|=q_1$, and $|E(G_2)|=q_2$. Note that $G$ contains $m+n$ copies of $G_1$, let $G_1^i$ be the copy of $G_1$ correspond to the vertex $v_i$ in $G_2$ $(i\in[m])$, which have partition $(X_1^i,Y_1^i)$, and let $G_1^j$ be the copy of $G_1$ correspond to the vertex $v_j$ in $G_2$ $(j\in\{m+1,\cdots,m+n\})$, which have partition $(X_1^j,Y_1^j)$. If $v_i$ is adjacent to $v_j$, connect the corresponding points of $G_1^i$ and $G_1^j$, and denote $L_{i,j}$ as the set of edges connecting $G_1^i$ and $G_1^j$. Note that $G$ is bipartite with partition $(X,Y)$ with $X=(\cup_{1\leq i\leq m}X_1^i)\cup(\cup_{m+1\leq j\leq m+n}Y_1^j)$ and $Y=(\cup_{1\leq i\leq m}Y_1^i)\cup(\cup_{m+1\leq j\leq m+n}X_1^j)$.

Let $AT(G_1)=\lceil\frac{q_1}{p}\rceil+1=k_1,AT(G_2)=\lceil\frac{q_2}{m+n}\rceil+1=k_2$, by the definition of the Alon-Tarsi number, the graph $G_1$ has an $AT$-orientation $D_1$, such that every vertex $x\in V(D_1)$ with outdegree at most $k_1-1$. Similarly, the graph $G_2$ has an $AT$-orientation $D_2$, such that every vertex $x\in V(D_2)$ with outdegree at most $k_2-1$. Then we give an orientation $D$ for $G$ by the following rules:

$R_1:$ For the copy $G_1^i$ and $G_1^j$, the edges belonging to them are oriented as $D_1$.

$R_2:$ The edges in $L_{i,j}$ are oriented as $D_2$, that is to say, oriented from $G_1^i$ to $G_1^j$ if $v_iv_j$ is oriented from $v_i$ to $v_j$.

By Lemma 2.2, the orientation $D$ of $G$ is an $AT$-orientation, and it has outdegree at most $k_1+k_2-2$, hence $AT(G)\leq AT(G_1)+AT(G_2)-1$.
It is easy to know that $|V(G)|=p(m+n)$ and $|E(G)|=q_1(m+n)+pq_2$. Since $\sum_{x\in V(D)}d^+_D(x)=|A(D)|$, by the Pigeonhole Principle, there exists some vertices have outdegree at least $\lceil\frac{q_1}{p}+\frac{q_2}{m+n}\rceil$ for any orientation $D$ of $G$,  $AT(G)\geq\lceil\frac{q_1}{p}+\frac{q_2}{m+n}\rceil+1$, obviously, $AT(G)\geq\lceil\frac{q_1}{p}\rceil+\lceil\frac{q_2}{m+n}\rceil+1=AT(G_1)+AT(G_2)-1$. Hence $AT(G)=AT(G_1)+AT(G_2)-1$. (See Figure 2 for an orientation of $C_4\square T_4$)
\end{proof}
\begin{figure}[htbp]
\centering
\includegraphics[height=11cm, width=0.8\textwidth]{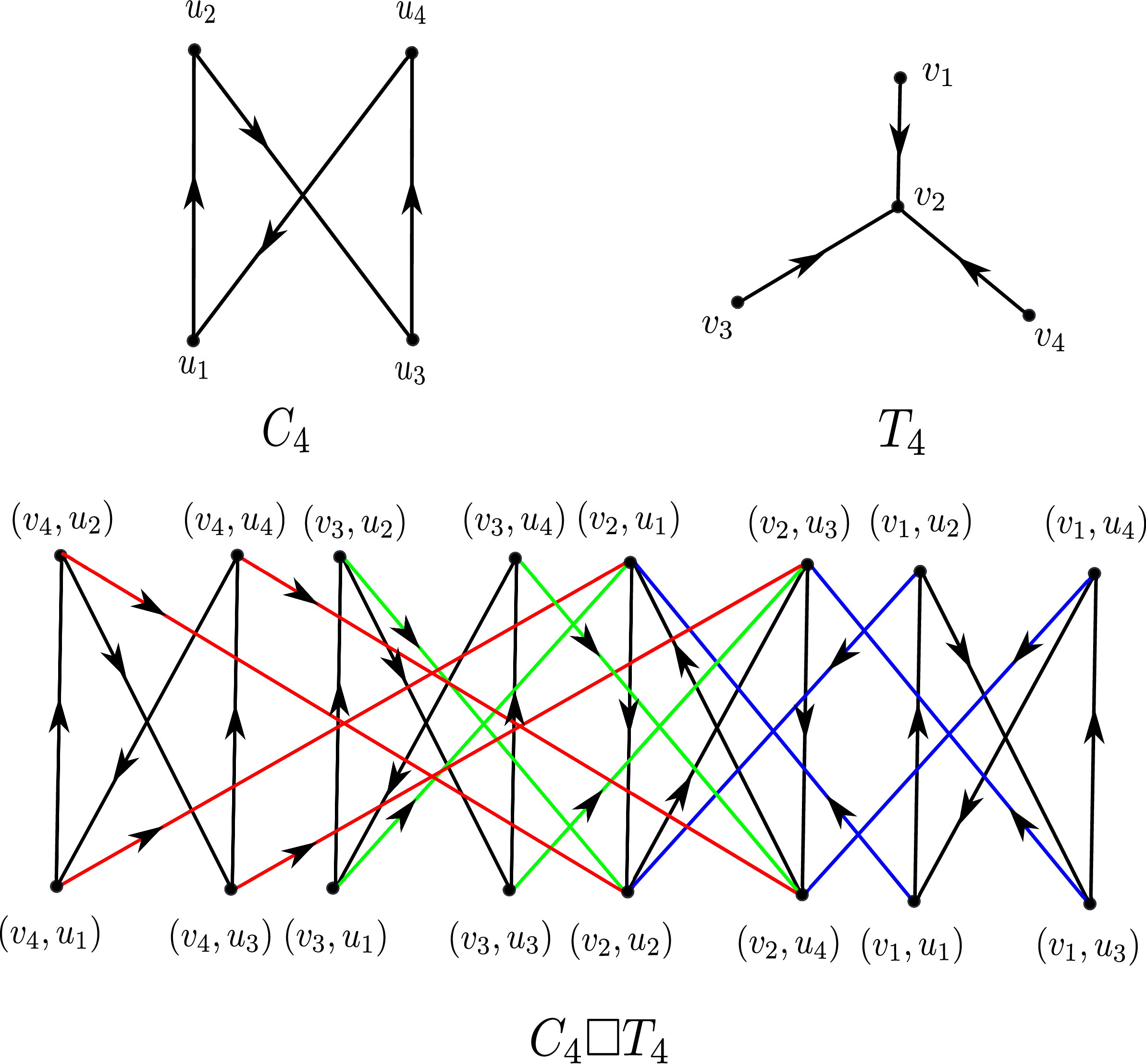}
\centerline{Fig.2 An orientation of $C_4\square T_4$. }
\end{figure}
\noindent\textbf{Proof of Theroem 1.}
Since $T_m$ is a tree, $AT(T_m)=2=\lceil\frac{|E(T_m)|}{|V(T_m)|}\rceil+1=\lceil\frac{m-1}{m}\rceil+1$. It is easy to see that $\frac{|E(Q_n)|}{|V(Q_n)|}=\frac{n}{2}$, by Lemma 3.2, $AT(Q_{n})=\lceil\frac{n}{2}\rceil+1=\lceil\frac{|E(Q_n)|}{|V(Q_n)|}\rceil+1$. 

When $n$ is even, it is obvious that $\lceil\frac{n}{2}\rceil=\frac{n}{2}$, and $\lceil\frac{n}{2}+\frac{m-1}{m}\rceil=\frac{n}{2}+1=\lceil\frac{n}{2}\rceil+\lceil\frac{m-1}{m}\rceil$.
By Lemma 3.3, $AT(Q_{n}\square T_m)=\lceil\frac{n}{2}\rceil+2$.

When $n$ is odd and $m=2$, obviously, $\lceil\frac{n}{2}+\frac{m-1}{m}\rceil=\lceil\frac{n+1}{2}\rceil=\lceil\frac{n}{2}\rceil\neq\lceil\frac{n}{2}\rceil+\lceil\frac{m-1}{m}\rceil$. Since $Q_{n}\square T_2=Q_{n}\square K_2=Q_{n+1}$,  $AT(Q_{n}\square T_2)=AT(Q_{n+1})=\lceil\frac{n+1}{2}\rceil+1=\lceil\frac{n}{2}\rceil+1$.

When $n$ is odd and $m\textgreater2$, we have $\lceil\frac{n}{2}+\frac{m-1}{m}\rceil=\frac{n+1}{2}+1=\lceil\frac{n}{2}\rceil+\lceil\frac{m-1}{m}\rceil$. By Lemma 3.3, $AT(Q_{n}\square T_m)=\lceil\frac{n}{2}\rceil+2$. (See Figure 2 for an orientation of $Q_2\square T_4$)

The result is established.

\begin{coro}
For any $n,k\in \mathbb{N}$, let $C_{2k}$ be an even cycle. Then $AT(Q_{n}\square C_{2k})=\lceil\frac{n}{2}\rceil+2.$
\end{coro}
\begin{proof}
It is easy to know that $\frac{|E(C_{2k})|}{|V(C_{2k})|}=1,~AT(C_{2k})=\lceil\frac{|E(C_{2k})|}{|V(C_{2k})|}\rceil+1$, and $\lceil\frac{n}{2}+1\rceil=\lceil\frac{n}{2}\rceil+1$, by Lemma 3.3, $AT(Q_{n}\square C_{2k})=\lceil\frac{n}{2}\rceil+2$.
\end{proof}
\noindent\textbf{Remark.}
The $Q_{n}\square T_m$ and $Q_{n}\square C_{2k}$ are not $chromatic$-$AT~choosable$.

Next we prove Theorem 2 by the following lemmas.

Let $G_1$ and $G_2$ be simple undirected graphs, assume that   $V(G_1)=\{v_1,\cdots,v_m\}$. If $G$ is the corona $G_1\circ G_2$ of $G_1$ and $G_2$, by the Definition 2.2, note that $G$ contains one copy of $G_1$ and $m$ disjoint copies of $G_2$, let $G_2^i$ be the copy of $G_2$ connect the vertex $v_i$ in $G_1$ $(i\in[m])$.
\begin{lem}
If $G$ is the corona $G_1\circ G_2$ of $G_1$ and $G_2$, then
\[ \chi(G)=\begin{cases}
		\chi(G_1) ,  & \chi(G_2)\textless\chi(G_1);\\
		\chi(G_2)+1,  &  \chi(G_2)\geq\chi(G_1).
	\end{cases} \]
\end{lem}
\begin{proof}
The graph $G_1$ is the subgraph of $G$, hence $\chi(G)\geq\chi(G_1)$. Since $v_i$ is adjacent to each vertex of the $G_2^i$, we can know that $\chi(G)\geq\chi(G_2)+1$. It is easy to know that $\chi(G)\leq\chi(G_1)$ if $\chi(G_2)\leq\chi(G_1)-1$, and $\chi(G)\leq\chi(G_2)+1$ if $\chi(G_1)\leq\chi(G_2)$. So $\chi(G)=\chi(G_1)$ when $\chi(G_2)<\chi(G_1)$ (see Figure 3 for a proper coloring of $C_3\circ C_4$), and $\chi(G)=\chi(G_2)+1$ when $\chi(G_2)\geq\chi(G_1)$ (see Figure 3 for a proper coloring of $C_4\circ C_3$).
\begin{figure}[htbp]
	\centering
	\includegraphics[height=7cm, width=0.8\textwidth]{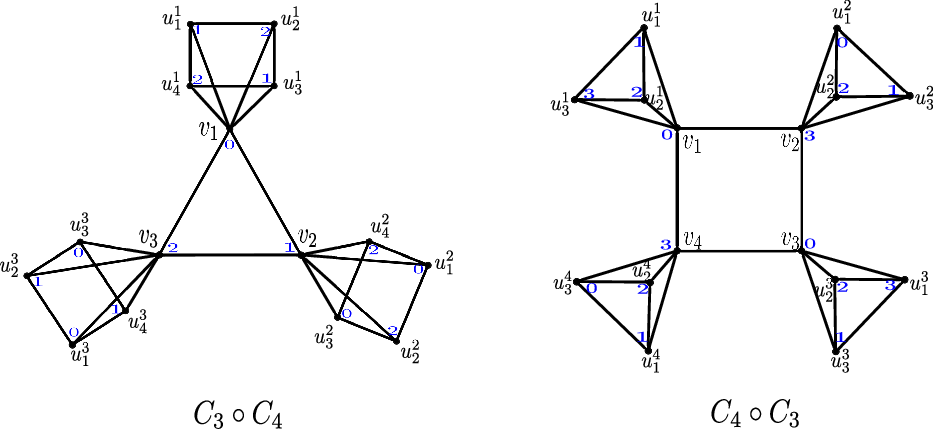}
	\centerline{Fig.3 A proper coloring of $C_3\circ C_4$ and  $C_4\circ C_3$. }
\end{figure}
\end{proof}

\begin{lem}
Let $G$ is the corona $G_1\circ G_2$ of $G_1$ and $G_2$, then
\[ AT(G)=\begin{cases}
	AT(G_1) ,  & AT(G_2)\textless AT(G_1);\\
	AT(G_2)~ or~ AT(G_2)+1,  &  AT(G_2)\geq AT(G_1).
\end{cases} \]
\end{lem}
\begin{proof}
Assume that $AT(G_1)=k_1,AT(G_2)=k_2$, by the definition of the Alon-Tarsi number, the graph $G_j$ has an $AT$-orientation $D_j$ $(j=1,2)$, such that every vertex $x\in V(D_j)$ with outdegree at most $k_j-1$. Denote $L_i$ as the set of edges connecting $G_2^i$ and $v_i~ (i\in[m])$. We give  an orientation $D$ for $G$ by the following rules:

$R_1:$ For the copy $G_1$, the edges belonging to $G_1$ are oriented as $D_1$.

$R_2:$ For the copies $G_2^i$, the edges belonging to $G_2^i$ are oriented as $D_2$.

$R_3:$ The edges in $L_i$ are oriented from $G_2^i$ to $v_i$.

By Lemma 2.2, the orientation $D$ of $G_1\circ G_2$ is an $AT$-orientation, and it has outdegree at most $\max\{k_1-1,k_2\}$, hence $AT(G)\leq \max\{AT(G_2)+1,AT(G_1)\}$. The graph $G_1$ and $G_2$ are the subgraphs of $G_1\circ G_2$, and $\chi(G)\leq AT(G)$, so $AT(G)\geq \max\{AT(G_1),AT(G_2),\chi(G)\}$.

If $AT(G_2)< AT(G_1)$, then $AT(G_1)\leq AT(G)\leq AT(G_1)$, $AT(G)=AT(G_1)$. If $AT(G_2)\geq AT(G_1)$, then $AT(G_2)\leq AT(G)\leq AT(G_2)+1$.
\end{proof}

\begin{coro}
Let $G$ is the corona $G_1\circ G_2$ of $G_1$ and $G_2$, and if $AT(G_2)\geq AT(G_1)$ and $G_2$ is $chromatic$-$AT~choosable$, then $G$ is $chromatic$-$AT~choosable$.
\end{coro}
\begin{proof}
	If $G_2$ is $chromatic$-$AT~choosable$, then $\chi(G_2)=AT(G_2)$. Since $AT(G_2)\geq AT(G_1)$, $\chi(G_2)\geq AT(G_1)\geq\chi(G_1)$, by Lemma 3.5, $\chi(G)=\chi(G_2)+1=AT(G_2)+1$, $AT(G)=AT(G_2)+1=\chi(G)$, then $G$ is $chromatic$-$AT~choosable$.
\end{proof}

\noindent\textbf{Proof of Theroem 2.}
By Lemma 3.2, $AT(Q_n)\geq 3$ when $n\textgreater2$. We have $AT(G_2)< AT(Q_n)$, by Lemma 3.6, $AT(Q_{n}\circ G_2)=\lceil\frac{n}{2}\rceil+1$ when $n\textgreater2$.

When $n\leq2$. It is easy to see that $Q_{n}\circ G_2$ contains a triangle as its subgraph, we have $AT(Q_{n}\circ G_2)\geq3$. By Lemma 3.6, $AT(Q_{n}\circ G_2)\leq \max\{AT(Q_n),AT(G_2)+1\}=3$, hence $AT(Q_{n}\circ G_2)=3$. (See Figure 4 for an orientation of $Q_2\circ T_4$)
\begin{figure}[htbp]
\centering
\includegraphics[height=7cm, width=0.5\textwidth]{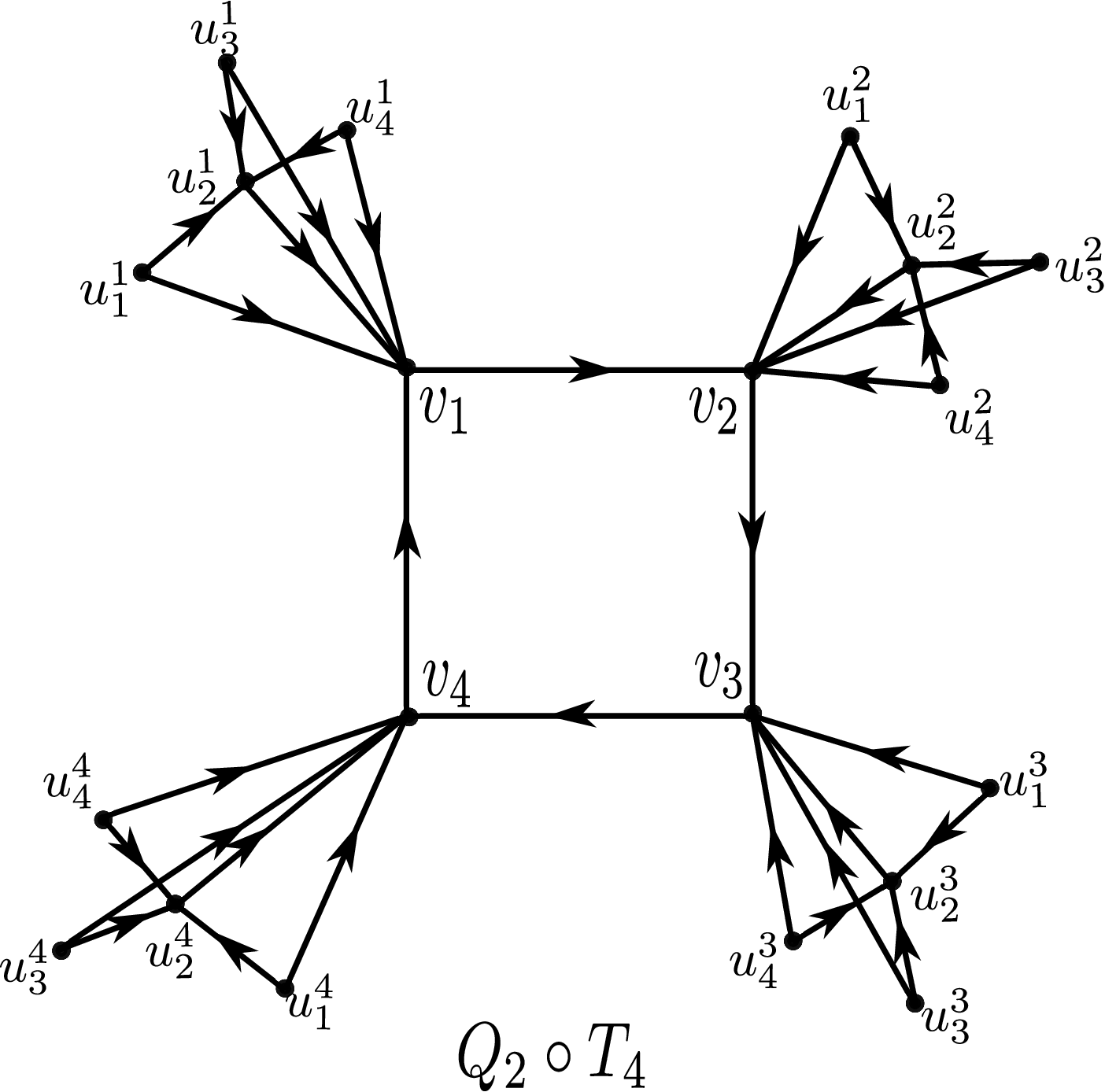}
\centerline{Fig.4 An orientation of $Q_2\circ T_4$. }
\end{figure}

\begin{coro}
For any $n,m,k\in \mathbb{N}$, and let $C_{2k}$ be an even cycle. Then
\[ AT(Q_{n}\circ C_{2k})=\begin{cases}
		3 ,  & n\leq2;\\
		\lceil\frac{n}{2}\rceil+1,  &  n\textgreater2.
	\end{cases} \]
\end{coro}

\begin{lem}
For any $n,k\in \mathbb{N}$, let $C_{2k+1}$ be an odd cycle. Then
 \[ AT(Q_{n}\circ C_{2k+1})=\begin{cases}
		4 ,  & n\leq4;\\
		\lceil\frac{n}{2}\rceil+1,  &  n\textgreater4.
	\end{cases} \]
\end{lem}
\begin{proof}
Since $\chi(Q_n)=2<\chi(C_{2k+1})=3$, by Lemma 3.5, we have $\chi(Q_{n}\circ C_{2k+1})=\chi(C_{2k+1})+1=4$. By Lemma 3.2, $AT(Q_n)\geq 4$ when $n\textgreater4$, $AT(C_{2k+1})< AT(Q_n)$, by Lemma 3.6, we have $AT(Q_{n}\circ C_{2k+1})=\lceil\frac{n}{2}\rceil+1$ when $n\textgreater4$.

When $n\leq4$. By Lemma 3.6, $AT(Q_{n}\circ C_{2k+1})\leq \max\{AT(Q_n),AT(C_{2k+1})+1\}=4$, and $AT(Q_{n}\circ C_{2k+1})\geq\chi(Q_{n}\circ C_{2k+1})=4$, hence $AT(Q_{n}\circ C_{2k+1})=4$.
\end{proof}

\noindent\textbf{Remark.}
The $Q_{n}\circ G_2$ and $Q_{n}\circ C_m$ are not $chromatic$-$AT~choosable$.

\section{Conclusion}
In this article, we obtain the exact value of the Alon-Tarsi number of $n$-cube $Q_n$; $Q_{n}\square T_m$ and $Q_{n}\square C_{2k}$; and $Q_{n}\circ G_2$, where $AT(G_2)=2$, and $Q_{n}\circ C_m$. Additionally, note that $n$-cube $Q_n$,  $Q_{n}\square T_m$, $Q_{n}\square C_{2k}$;  and $Q_{n}\circ G_2$, $Q_{n}\circ C_m$ are not $chromatic$-$AT~choosable$. Corollary 3.7 shows that when $G_1$ and $G_2$ satisfy certain conditions, if $G_2$ is $chromatic$-$AT~choosable$, then $G_1\circ G_2$ is $chromatic$-$AT~choosable$. This leads to the following question.
\\

\noindent\textbf{Question.} 
When $G_1$ and $G_2$ satisfy certain conditions, whether $G_1$ or $G_2$ is $chromatic$-$AT~choosable$, then $G_1\circ G_2$ is $chromatic$-$AT~choosable$?
\section{Acknowledgement}
This work was partially funded by Science and Technology Project of Hebei Education Department, China (No. ZD2020130) and the Natural Science Foundation of Hebei Province, China (No. A2021202013)

\end{document}